\documentclass[reqno, 12pt]{amsproc}
\usepackage{amsmath,amsthm,amsfonts,amssymb}
\usepackage[cp1251]{inputenc}
   \usepackage[english]{babel}
\date{}
\newtheorem{theorem}{Theorem}[section]
\newtheorem{proposition}[theorem]{Proposition}%[section]
\newtheorem{example}[theorem]{Example}%[section]
\numberwithin{equation}{section}

\begin{document}

\title[Infinite-dimensional  Fokker--Planck--Kolmogorov equations]{Infinite-dimensional nonlinear  stationary  Fokker--Planck--Kolmogorov equations}

\author[V. Bogachev,  M. R\"ockner,  S. Shaposhnikov]{Vladimir I. Bogachev,  Michael R\"ockner,
Stanislav V. Shaposhnikov}
\address{V.B., S.S.: Moscow Lomonosov State University,
National Research University ``Higher School of Economics'';
M.R.: Bielefeld University, Germany}

\begin{abstract}
We prove existence of a probability solution to the nonlinear stationa\-ry Fokker--Planck--Kolmogorov equation
on an infinite dimensional space with a cente\-red Gaussian measure $\gamma$ with a unit diffusion operator
and a drift of the form $-x+v(p,x)$, where $v$ is a bounded mapping with values in the Cameron--Martin space $H$
of $\gamma$ and $v$ is defined on the space $E\times X$, where is $E$ is the subset of $L^2(\gamma)$ consisting
of probability densities. The equation has the form $L_{b(p,\bullet)} ^*(p\cdot \gamma)=0$ with
$L_{b(p,\bullet)}\varphi =\Delta_H \varphi + (b(p,\bullet) , D_{_H}\varphi)_{_H}$,
so that the drift coefficient depends on the unknown solution, which makes the equation nonlinear. This dependence is assumed
to satisfy a suitable continuity condition.
This result is applied to drifts of Vlasov type defined by means of the convolution of a vector field with the solution.
In addition, we consider a more general situation where only the components of $v$ are uniformly bounded and prove
the existence of a probability solution under some stronger continuity condition on the drift.

\end{abstract}

\dedicatory{Dedicated to Leonard Gross on the occasion of his 95th birthday}

\maketitle

\section{Introduction}

Leonard Gross, one of the pioneers of infinite-dimensional analysis, initiated his program
of the study of infinite-dimensional harmonic analysis in his seminal papers \cite{G60}, \cite{G63}, and \cite{G67}.
In the subsequent years, this direction was developed by Gross himself, his students, including Piech, Kuo, and Gordina
(see, e.g., \cite{GRT}, \cite{Kuo}, \cite{KuoP}, \cite{Piech74}, and \cite{Piech75}), and many followers all over the world.
Part of Gross's program was the investigation of elliptic equations on infinite-dimensional spaces, see, e.g.,
\cite{Cer}, \cite{DP}, \cite{DPZ}, and \cite{DF}  (of course, there is also
a parabolic counterpart which is not discussed here). As in the finite-dimensional case,
several different types of elliptic equations arise: direct equations such as $Lf=\Delta f+\langle b,\nabla f\rangle =0$,
divergence type equations, and double-divergence type equations or adjoint equations,
which in case of a constant second order coefficients look like $\Delta f- {\rm div}(fb)=0$.
The latter equation is naturally defined for measures as
$$
L^*\mu=0,
$$
which in the finite-dimensional case is understood as the identity
$$
\int_{\mathbb{R}^n} L\varphi\, d\mu=0, \quad \varphi\in C_0^\infty(\mathbb{R}^n),
$$
 and in such a form is called the stationary Fokker--Planck--Kolmogorov equation.
 This equation is meaningful if $b$ is locally integrable with respect to $\mu$, in particular, if $b$ is locally bounded.

A nonlinear equation arises when the drift $b$ depends on the unknown solution~$\mu$.
For example, a Vlasov-type equation corresponds to the drift
$$
b(\mu,x)=\int_{\mathbb{R}^n}  b_0(x-y)\, \mu(dy).
$$
For a  recent survey of nonlinear Fokker--Planck--Kolmogorov equations see \cite{BS24}.
In this paper, we study nonlinear  stationary  Fokker--Planck--Kolmogorov equations on infinite-dimensional spaces.
Such equations are defined similarly, which is discussed in Section~2.

Our main result (stated and proved in Section~3)
gives the existence of a probabi\-lity solution to the infinite-dimensional nonlinear stationary equation with the drift
$$
b(\mu,x)=-x+v(\mu,x),
$$
 where $v$ takes values in the Cameron--Martin space $H$ of a centered Radon Gaussian measure $\gamma$
on a locally convex space~$X$ and is uniformly bounded with respect to the Cameron--Martin norm on~$H$ and
is continuous with respect to $\mu$ is a suitable sense. A~solution is constructed in the set of probability measures with
densities from~$L^2(\gamma)$. It is also shown that if the field $v_u$ depends Borel measurably on a parameter $u$ from a
complete separable metric space~$U$, then one can select a solution $p_u$ that is Borel measurable in~$u$.
Finally, we consider a more general case where only the components of $v$ are uniformly bounded. However, in order to ensure
the existence of a probability solution in this case we impose some stronger continuity condition on~$v$.

\section{Notation, terminology and auxiliary results}

Recall (see \cite{B07}) that a Borel probability measure on a topological space is called Radon if for every Borel set $B$ we have
$\mu(B)=\sup \mu(K)$, where $\sup$ is taken over compact subsets of $B$.

Let $\gamma$ be a centered Radon Gaussian measure on a locally convex space $X$ (see, e.g.,~\cite{B98}
for a detailed presentation of the theory of Gaussian measures), i.e., a~Radon probability measure such that
every continuous linear functional on $X$ is a centered Gaussian random variable with respect to~$\gamma$.
The space of continuous linear functionals on $X$ is denoted by~$X^*$.
Without loss of generality
one can assume that $X=\mathbb{R}^\infty$ is the countable power of the real line
and $\gamma$ is the countable power of the standard Gaussian measure on the real line.
In that case the dual space $X^*$ can be identified with the space $\mathbb{R}^\infty_0$ of finite sequences.
This measure $\gamma$ is called the standard Gaussian measure on~$\mathbb{R}^\infty$.

Let $X_\gamma^*$ be the closure of $X^*$ in $L^2(\gamma)$.
For the standard Gaussian measure $\gamma$  on~$\mathbb{R}^\infty$ the space $X_\gamma^*$ can be identified with
the standard Hilbert space~$l^2$: each element $(h_n)$ of $l^2$ generates the element
$\widehat{h}(x)=\sum_{n=1}^\infty h_n x_n$ of $X_\gamma^*$, and conversely every element of $X_\gamma^*$ admits such representation.

The Cameron--Martin space $H=H(\gamma)$ of $\gamma$ is the space of all vectors $h$ with finite norm
$$
|h|_{_H}=\sup\biggl\{l(h)\colon l\in X^{*}, \int_X l^2\, d\gamma \le 1\biggr\}.
$$
For every element $h\in H(\gamma)$ there is a unique element $\widehat{h}\in X_\gamma^*$ such that
$$
l(h)=\int_X \widehat{h} l\, d\gamma \quad \forall\, l\in X^*.
$$
The norm on $H=H(\gamma)$ indicated above is generated by the inner product
$$
(u,v)_{_H}=\int_X \widehat{u}\widehat{v}\, d\gamma.
$$
The mapping $h\mapsto \widehat{h}$ is a linear isometry of $H(\gamma)$ and $X_\gamma^*$, in particular,
for each $f\in X_\gamma^*$ there is a vector $h\in H(\gamma)$ with $f=\widehat{h}$,  and the inverse isometry is denoted
by $j_{_H}$. Thus, $j_{_H}(\widehat{h})=h$ and
$$
(j_{_H}(f), j_{_H}(g))_{_H}=(f,g)_{L^2(\gamma)}\quad \forall\, f,g\in X_\gamma^*.
$$
For the  standard Gaussian measure $\gamma$  on~$\mathbb{R}^\infty$ the Cameron--Martin space coincides with $l^2$
and $h\mapsto \widehat{h}$ is the identity mapping when $X_\gamma^*$ is identified with $l^2$ as explained above.

It is well known that the Cameron--Martin space is a separable Hilbert space. In addition, the measure $\gamma$ is concentrated
on a countable union of metrizable compact sets, so one can assume below that $X$ is such a space. Moreover, if the space $H$ is infinite-dimensional,
then there is a Borel linear isomorphism between $\gamma$ and the standard Gaussian measure $\gamma_\infty$ in the following sense:
there are Borel linear subspaces $X_1\subset X$ and $X_2\subset \mathbb{R}^\infty$ with
$\gamma(X_1)=\gamma_\infty(X_2)=1$ and one-to-one Borel linear operator $J\colon X_1\to X_2$ with Borel $J^{-1}$ such that $J$ takes
$\gamma$ to $\gamma_\infty$ and $J\colon H\to l^2$ is an isometry. It follows that the problem we discuss below reduces to the case
of $\mathbb{R}^\infty$ with  the standard Gaussian measure. The readers who prefer to deal with coordinate representations can assume that
we consider this case.

It is always possible to find an orthonormal basis $\{e_n\}$ in $H(\gamma)$ such that
$\widehat{e}_n\in X^*$. For all $h\in H$ and $l\in X^*$ we
have
\begin{equation}\label{e1}
l(h)=(h, j_{_H}(l))_{_H}=\sum_{n=1}^\infty (h,e_n)_{_H} (j_{_H}(l),e_n)_{_H}=\sum_{n=1}^\infty \widehat{e}_n(h)l(e_n).
\end{equation}

According to Tsirelson's theorem (see \cite{B98}), for any orthonormal basis $\{e_n\}$ in $H(\gamma)$,
we have
$$
x=\sum_{n=1}^\infty \widehat{e}_n(x)e_n \quad \hbox{$\gamma$-a.e.},
$$
where the series converges in $X$. In particular, for every $l\in X^*$ we have
\begin{equation}\label{e2}
l(x)=\sum_{n=1}^\infty \widehat{e}_n(x)l(e_n) \quad \hbox{$\gamma$-a.e.}
\end{equation}

Let $\mathcal{F}\mathcal{C}$ be the space of functions on $X$ of the form
$$
\varphi(x)=\psi(l_1(x),\ldots,l_n(x)), \quad \psi\in C_b^\infty(\mathbb{R}^n), \ l_i\in X^*.
$$
Such functions will be called smooth cylindrical.
Any function $\varphi\in \mathcal{F}\mathcal{C}$  of this form  has bounded partial derivatives
$\partial_h \varphi$ for all $h\in X$ and
$$
\partial_h \varphi(x)=\lim\limits_{t\to 0} \frac{\varphi(x+th)-\varphi(x)}{t}=\sum_{i=1}^n \partial_{x_j}\psi(l_1(x),\ldots,l_n(x))l_j(h).
$$
Consequently,
$$
\partial_h^2 \varphi=\sum_{i,j=1}^n \partial_{x_j}\partial_{x_i}\psi(l_1,\ldots,l_n)l_i(h)l_j(h).
$$

Suppose that $v\colon X\to H$ is a Borel mapping and
$$
b(x)=-x+v(x).
$$
 Let us fix an orthonormal basis $\{e_n\}$ in $H(\gamma)$
such that $\widehat{e}_n\in X^*$.
Then we can define the operator
$$
L_b\varphi(x) :=\sum_{i=1}^\infty [\partial_{e_i}^2\varphi(x) + \widehat{e}_i(b(x))\partial_{e_i}\varphi(x)]
$$
defined on $\mathcal{F}\mathcal{C}$ and taking values in $L^2(\gamma)$.
Indeed, writing $\varphi$ as above, we have
\begin{multline*}
\partial_{e_i}^2\varphi + \widehat{e}_i(b(x))\partial_{e_i}\varphi(x)
\\
=
\sum_{j,k\le n} \partial_{x_k}\partial_{x_j}\psi(l_1,\ldots,l_n)l_k(e_i)l_j(e_i)
+\sum_{j\le n} \widehat{e}_i(b(x))\partial_{x_j}\psi(l_1,\ldots,l_n)l_j(e_i),
\end{multline*}
which gives
\begin{multline*}
L_b\varphi =\sum_{j,k\le n} \partial_{x_k}\partial_{x_j}\psi(l_1,\ldots,l_n)\sum_{i=1}^\infty l_k(e_i)l_j(e_i)
\\
+ \sum_{j\le n} \partial_{x_j}\psi(l_1,\ldots,l_n) \sum_{i=1}^\infty \widehat{e}_i(b(x)) l_j(e_i),
\end{multline*}
where
$$
\sum_{i=1}^\infty l_k(e_i)l_j(e_i)=(j_{_H}(l_k), j_{_H}(l_j))_{_H},
$$
$$
\sum_{i=1}^\infty \widehat{e}_i(b(x)) l_j(e_i)=-\sum_{i=1}^\infty \widehat{e}_i(x)l_j(e_i)+\sum_{i=1}^\infty
\widehat{e}_i(v(x)) l_j(e_i).
$$
According to \eqref{e1} and \eqref{e2} the latter equals $\gamma$-a.e.
$$
-l_j(x)+ l_j(v(x))=l_j(b(x)).
$$
Therefore,
$$
L_b\varphi=\sum_{j,k\le n} \partial_{x_k}\partial_{x_j}\psi(l_1,\ldots,l_n)(j_{_H}(l_k), j_{_H}(l_j))_{_H}
+\sum_{j\le n} l_j(b(x)).
$$
Thus, $L_b\varphi\in L^2(\gamma)$ and $L_b$ does not depend on our choice of $\{e_n\}$ with the indicated properties.

A Radon probability measure $\mu$ absolutely continuous with respect to $\gamma$ satisfies the stationary
Fokker--Planck--Kolmogorov equation
\begin{equation}\label{eq1}
L_b^{*}\mu=0
\end{equation}
if $l(b)\in L^1(\mu)$ for all $l\in X^*$ and
\begin{equation}\label{eq2}
\int_X L_b\varphi \, d\mu=0 \quad \forall \varphi\in \mathcal{F}\mathcal{C}.
\end{equation}

It is worth noting that \eqref{eq1} can be interpreted in a weaker sense: we fix an orthonormal basis $\{e_n\}$ in $H(\gamma)$
 such that $\widehat{e}_n\in X^*$, consider the class $\mathcal{F}\mathcal{C}_{\{e_n\}}$  of smooth cylindrical
 functions of the form indicated above defined by means of the sequence of functionals $l_n=\widehat{e}_n$
 and define \eqref{eq1} by means of the identity  \eqref{eq2} on $\mathcal{F}\mathcal{C}_{\{e_n\}}$.
 However, it is known (see \cite{BR95} or \cite[Theorem~7.5.6]{B98}) that if $v$ is bounded, then
 any Borel probability measure $\mu$ with $X^*\subset L^2(\mu)$ satisfying the equation
 $L_b^{*}\mu=0$ in this weaker sense is absolutely continuous with respect to $\gamma$,
 moreover, its Radon--Nikodym density belongs to $L^2(\gamma)$ (see \cite{BKS23}). Note also that in this case
 $\varrho=d\mu/d\gamma$ belongs to the Gaussian Sobolev class $W^{1,1}(\gamma)$ and
 $$
 \int_ X \frac{|D_{_H}\varrho|_{_H}^2}{\varrho}\, d\gamma \le \int_X |b|_{_H}^2\, d\gamma.
 $$

 A more general situation is studied in \cite{BSS19}, where the following theorem is proved.
 Let us consider the  class $\mathcal{F}\mathcal{C}_{0,\{e_n\}}$  of smooth cylindrical
 functions of the form
 $$\varphi(x)=\psi(l_1(x),\ldots,l_n(x))$$
  with $l_n=\widehat{e}_n$ and functions
 $\psi\in C_0^\infty(\mathbb{R}^n)$; unlike $\mathcal{F}\mathcal{C}$, this class is not a linear space.
 If $X=\mathbb{R}^\infty$ and $H=l^2$ with its natural basis $\{e_n\}$, then $l_n=\widehat{e}_n$ is the $n$-coordinate
 function and $\mathcal{F}\mathcal{C}_{0,\{e_n\}}$ is just the union of all $C_0^\infty(\mathbb{R}^n)$.

 A Borel measure  $\mu$ is called a solution to the equation
 $L_b^*\mu=0$ with respect to $\mathcal{F}\mathcal{C}_{0,\{e_n\}}$  if
 $l_n(v)\in L^1(\mu)$ for all $n$ and
 \eqref{eq2} holds for all $\varphi\in \mathcal{F}\mathcal{C}_{0,\{e_n\}}$.
Although $\mathcal{F}\mathcal{C}_{0,\{e_n\}}$ is not a linear space, an advantage
of using this class is that for any function $\varphi\in \mathcal{F}\mathcal{C}_{0,\{e_n\}}$
the functions $l_j\partial_{e_j} \varphi$ are bounded, so in the case of bounded $l_j(v)$ the functions
$L_b\varphi$ are also bounded and belong to $L^1(\mu)$.

 \begin{theorem}
Let $\mu$ be  a Borel probability measure on~$X$ such that $|v|_{_H}\in L^1(\mu)$
and $L_b^*\mu=0$ with respect to $\mathcal{F}\mathcal{C}_{0,\{e_n\}}$, where $b(x)=-x+v(x)$. Then $\mu$ is absolutely continuous with respect to $\gamma$
and for its Radon--Nikodym $f:=d\mu/d\gamma$ we have
\begin{equation}\label{b2}
\int_X f \bigl(\log(f+1)\bigr)^{\alpha}\, d\gamma \le
C(\alpha)\Bigl[
1 + \bigl\| |v|_H \bigr\|_{L^1(\mu)}
\Bigl(\log\bigl(1 + \bigl\| |v|_H \bigr\|_{L^1(\mu)}\bigr)\Bigr)^{\alpha}\Bigr]
\end{equation}
for every  $\alpha<1/4$, where $C(\alpha)$ is a number depending only on $\alpha$.
\end{theorem}

It follows from this result that the solution with respect to $\mathcal{F}\mathcal{C}_{0,\{e_n\}}$ will be also a
solution with respect to $\mathcal{F}\mathcal{C}$, because any function $f\in \mathcal{F}\mathcal{C}$
can be approximated in any $L^p(\gamma)$ by functions from  $\mathcal{F}\mathcal{C}_{0,\{e_n\}}$
along with its partial derivatives $\partial_{e_i}f$, $\partial_{e_i}^2f$, where $i\le N$ and $N$ is fixed.

Solutions to nonlinear equations are defined similarly.
Namely, if $E$ is a Borel set of probability densities in $L^1(\gamma)$ and
$v\colon E\times X\to H$ is a Borel mapping, then a measure $\mu=\varrho\cdot \gamma$ with $\varrho\in E$
is called a solution to the nonlinear Fokker--Planck--Kolmogorov equation
\begin{equation}\label{eq1-n}
L_{b(\varrho, \bullet)}^*\mu=0, \quad \hbox{where\ } b(\varrho,x)=-x+v(\varrho,x),
\end{equation}
if the measure $\mu$ satisfies the linear equation \eqref{eq1} with $b=b(\varrho, \bullet)$.

It is worth noting that the estimate from this theorem is related to the logarithmic Sobolev inequality,
another important contribution of Leonard Gross (see, e.g., \cite{G75} and~\cite{G25}), which, however, 
should be a subject of a separate survey.

\section{Main results}

Our main result states the existence of a probability solution to the nonlinear equation with a bounded field $v$.

\begin{theorem}
Let $v\colon E\times X\to H$ be a bounded Borel mapping, where $E$ is the subset of $L^2(\gamma)$ consisting of probability densities,
equipped with the weak topology of~$L^2(\gamma)$.
Suppose that for each $x$ and each $h\in H$ the function $p\mapsto (v(p,x), h)_{_H}$
is sequentially continuous on $E$ with respect to the weak topology of $L^2(\gamma)$. Let $b(p,x)=-x+v(p,x)$.
Then  equation \eqref{eq1-n} has a solution $\mu$ given by a density $p_\mu$ from~$E$.
\end{theorem}
\begin{proof}
Let us fix $p\in E$. According to \cite{Shig}, there is a unique probability density $\varrho_p\in L^1(\gamma)$ such that the measure
$\varrho\cdot \gamma$ satisfies the equation
$L_{b(p,\bullet)}^* (\varrho_p\cdot \gamma)=0$ with the drift $b(p,x)=-x+v(p,x)$.
It follows from \cite{Hino} and \cite{BKS23} that $\varrho_p\in L^2(\gamma)$.
Therefore, we obtain a mapping $\Psi\colon p\mapsto \varrho_p$ from $E$ to~$E$.
We are going to apply Schauder's theorem to show that this mapping has a fixed point. Clearly,
any fixed point gives a solution to the nonlinear equation.

First we observe that there is a constant $C$ such that every probability solution to the equation
$L_{b(p,\bullet)}^* (\varrho\cdot \gamma)=0$ with $p\in E$ satisfies the estimate
\begin{equation}\label{b1}
\int_X \varrho^2\, d\gamma \le C.
\end{equation}
This follows from \cite[Theorem~2.3]{BKS23}, which gives the bound
$$
\gamma(x\colon \varrho(x)\ge t)\le e^2 \exp \big(-\sigma_\infty |\ln t|^2\big), \ t>1,
$$
where $\sigma_\infty =\big(2\pi \big\| |v(p,\bullet)|_{_H} \big\|_{\infty} \big)^{-2}$.
If $|v(p,\bullet)|_{_H}=0$, then this bound means that $\varrho=1$. Once $|2\pi v(p,x)|_{_H}\le C_0$, we obtain
$$
\gamma(x\colon \varrho(x)\ge t)\le e^2 \exp \big(-C_0^{-2} |\ln t|^2\big), \ t>1,
$$
so that
$$
\int_X \varrho^2\, d\gamma \le 1+ 2e^2 \int_1^{+\infty} t \exp \big(-C_0^{-2} |\ln t|^2\big)\, dt.
$$
Next, the subset $S$ of $E$ satisfying \eqref{b1} is weakly compact in $L^2(\gamma)$.
In addition, this subset is convex.
In order to apply Schauder's theorem, it remains to verify that $\Psi$ is continuous on $S$ with respect to the
weak topology of $L^2(\gamma)$. Note that the weak topology is metrizable on~$S$, since $S$ is bounded and $L^2(\gamma)$ is separable.
Suppose that functions $p_n$ converge in $S$ to a function $p\in S$ in the weak topology. We have to show that the functions
$\Psi(p_n)$ converge weakly to~$\Psi(p)$.
Otherwise there is a subsequence $\{p_{n_k}\}$ such that $\Psi(p_{n_k})$  converges weakly to some
$g\in S$ different from $\Psi(p)$. Thus, we can assume that the whole sequence $\{\Psi(p_n)\}$ converges to $g\not=\Psi(p)$.
It suffices to show that $g$ satisfies the equation
$L_{b(p,\bullet)}^* (g\cdot \gamma)=0$, because $\Psi(p)$ is the only solution to this equation in~$S$.
Let $\varphi$ be a smooth cylindrical function
of the form $\varphi(x)=\psi(l_1(x),\ldots, l_m(x))$, where $\psi\in C_b(\mathbb{R}^m)$ and $l_i\in X^*$.
As explained above,
we have
$$
L_{b(p,\bullet)}\varphi=\sum_{j,k\le m} \partial_{x_k}\partial_{x_j}\psi(l_1,\ldots,l_n)(j_{_H}(l_k), j_{_H}(l_j))_{_H}
-\sum_{j\le m} l_j(x)+\sum_{j\le m} l_j(v(p,x))
$$
and similarly for $p_n$. Clearly,
$$
\int_X \partial_{x_k}\partial_{x_j}\psi(l_1,\ldots,l_n)\, \Psi(p_n)\, d\gamma \to
\int_X \partial_{x_k}\partial_{x_j}\psi(l_1,\ldots,l_n)\, g\, d\gamma,
$$
$$
\int_X l_j \Psi(p_n)\, d\gamma \to \int_X l_j g\, d\gamma
$$
for all $j\le m$. In addition,
$l_j(v(p_n,x))\to l_j(v(p,x))$ and these functions are uniformly bounded.
Hence
$$
\int_X l_j(v(p_n,x)) \Psi(p_n)(x)\, \gamma(dx) \to \int_X l_j(v(p,x)) g(x)\, \gamma(dx).
$$
Thus,
$$
\int_X L_{b(p_n,\bullet)}\varphi \Psi(p_n)\, d\gamma \to \int_X L_{b(p,\bullet)}\varphi g\, d\gamma,
$$
where the left-hand sides vanish, so the right-hand side vanishes too, which means that
$L_{b(p,\bullet)}^* (g\cdot \gamma)=0$. The proof is complete.
\end{proof}

It is worth noting that the solution density $p_\mu$ belongs not only to $L^2(\gamma)$, but also
to all $L^\alpha(\gamma)$, $\alpha\in [1,  +\infty)$, which follows from \cite[Theorem~4.1]{Hino}.
Moreover, according to \cite{BKS23}, there is $\varepsilon>0$ such that
$\exp (\varepsilon |\log \max (p_\mu,1)|^2)\in L^1(\gamma)$. However, as an example in \cite{BKS23} shows,
it can happen that $\exp(\varepsilon p_\mu)\not\in L^1(\gamma)$ for all $\varepsilon>0$.

\begin{example}
\rm
Let $b_0\colon X\to H$ be a bounded Borel mapping and
$$
b(p,x)=\int_X b_0(x-y)p(y)\, \gamma(dy).
$$
Then $b$ satisfies the hypotheses of the theorem.
Indeed,
if functions $p_n$ converge in $E$ to a function $p$ in the weak topology of $L^2(\gamma)$,
then, for every $x$, we have
$$
\int_X b_0(x-y)p_n(y)\, \gamma(dy)\to \int_X b_0(x-y)p(y)\, \gamma(dy).
$$
Due to the boundedness of $b_0$ we have convergence $b(p_n,x)\to b(p,x)$ in~$H$. In particular, we have
$(b(p_n,x),h)_{_H}\to (b(p,x),h)_{_H}$ for every $h\in H$.
For every fixed $p$ the mapping $x\mapsto b(p,x)$ on $X$ is Borel measurable.
For every fixed $x$ the mapping $p\mapsto b(p,x)$ is continuous on $E$ in the weak topology.
Since the sets
$$E_N=\{f\in E\colon \|f\|_{L^2(\gamma)}\le N\}$$
 are compact metrizable in the weak topology,
we conclude that $b$ is jointly Borel measurable on each $E_N\times X$ (see \cite[Exercise~6.10.40]{B07}).
Then $b$ is jointly Borel measurable on all of $E\times X$.
\end{example}

The theorem proved above can be extended to more general unbounded fields $v$ and non-constant diffusion operators.
However, this requires more technicalities based on the results from \cite{BKR96}, \cite{BSS19}, \cite{BKS23},
 and \cite{Hino} and will  be considered separately.

 We now assume that $v$ depends additionally on a parameter $u$ from a complete separable metric space~$U$ and
 $$
 (p,x,u)\mapsto v_u(p,x), \ E\times X\times U\to H
 $$
 is a bounded Borel mapping.

 \begin{proposition}
 There are solutions $p_u\in E$ to the equations
 $$
 L^{*}_{b_u(p_u,\bullet)} (p_u \cdot\gamma)=0
 $$
  such that the mapping $u\mapsto p_u$ from $U$ to $E$
  is Borel measurable.
 \end{proposition}
 \begin{proof}
 Let $S$ be the weakly compact set in $E$ introduced in the proof of the theorem above.
 For every $u\in U$ the set $S_u$ of all solutions $p\in S$ to the nonlinear equation
 $L^{*}_{b_u(p,\bullet)} (p\cdot\gamma)=0$ is compact.
 Let us show that the
 $$
 W=\{(u,p)\colon u\in U, \, p\in S_u\}
 $$
 is Borel in $U\times E$. Once this is done, we can apply the classical measurable selection theorem, which
 gives a Borel mapping $F\colon U\to S$ such that $F(u)\in S_u$ for every $u\in U$ (see  \cite[\S35]{K}).

 As explained above, we can assume that there is a countable family of functionals $l_n\in X^*$ separating points in~$X$,
 so we can assume that the functionals $l_n=\widehat{e}_n$ play this role
 (or simply deal with $\mathbb{R}^\infty$ and take the coordinate functions).
 Then equation $L_b^*(p\cdot\gamma)$ with respect to $\mathcal{F}\mathcal{C}$ is equivalent to this equation
 with respect to $\mathcal{F}\mathcal{C}_{0,\{e_n\}}$ (see the discussion in Section~2). In turn, testing the latter
 reduces to  checking \eqref{eq2} for a suitable countable family $\{\varphi_j\}\subset \mathcal{F}\mathcal{C}_{0,\{e_n\}}$.
  Therefore, the equality $L^{*}_{b_u(p,\bullet)} (p\cdot\gamma)=0$ is equivalent
 to the countable system of relations
 $$
 \int_X \sum_{i=1}^N [\partial_{e_i}^2\varphi_j - l_i\partial_{e_i}\varphi_j +(v_u(p,\bullet), e_i)_{_H}]p\, d\gamma=0.
 $$
 Thus, it remains to observe that this integral
 defines a Borel function on $U\times S$. This is obvious for the  terms with partial derivatives. Let us
 consider the function
 $$
 \int_X (v_u(p,\bullet), e_i)_{_H}] p\, d\gamma.
 $$
 Using an orthonormal basis $\{\psi_j\}$ in $L^2(\gamma)$, we obtain the functions
 $$
\int_X p\psi_j\, d\gamma \int_X (v_u(p,\bullet), e_i)_{_H}] \psi_j\, d\gamma,
 $$
 where the first integral is continuous in $p$ and the second one is Borel measurable in~$u$,
 because the function $(u,x)\mapsto (v_u(p,x), e_i)_{_H}] \psi_j(x)$ is jointly measurable.
  \end{proof}

  Finally, we consider a more general situation where $v$ still takes values in $H$, but is not assumed to be
  uniformly bounded with respect to the norm of~$H$. We now assume instead
  that
  $$
  |(v(\mu,x), e_n)_{_H}|\le C
  $$
  for some constant $C$ and some orthonormal basis $\{e_n\}$ in~$H$.
  To simplify our presentation we assume that $X=\mathbb{R}^\infty$ and $H=l^2$ with its standard basis $\{e_n\}$, however,
  a more abstract formulation can be derived easily from what follows.
  Now no Gaussian measure is fixed, so the mapping $v$ is defined
  on $\mathcal{P}(X)\times X$, where $\mathcal{P}(X)$ is the space of all Borel probability measures on $X$ equipped
  with the weak topology (it is well known that this topology is metrizable by a complete separable metric,
  for example, by the Prohorov metric or by the Kantorovich--Rubinshtein metric).

 \begin{theorem}
 Suppose that the functions $(\mu,x)\mapsto v_n(\mu,x):=(v(\mu,x), e_n)_{_H}$
 are continuous on~$\mathcal{P}(X)\times X$ and uniformly bounded.
 Then there is a probability solution to equation \eqref{eq1-n} with respect to the class $\mathcal{F}\mathcal{C}_{0,\{e_n\}}$.
 \end{theorem}
 \begin{proof}
 Let us take any sequence of positive numbers $\alpha_n$ with
 $$
 T:=\sum_{n=1}^\infty \alpha_n<\infty.
 $$
 We construct our solution on the weighted Hilbert space $E\subset X$ of sequences with finite
 norm
 $$
 |x|_{_E}=\Big(\sum_{n=1}^\infty \alpha_n x_n^2\Big)^{1/2}.
 $$
 For every $k\in\mathbb{N}$ we define a $k$-dimensional
 mapping $v^k\colon \mathcal{P}(\mathbb{R}^k)\times \mathbb{R}^k\to \mathbb{R}^k$ by setting
 $v_n^k(\mu,x)=v_n(\mu, x)$, $n\le k$, where each vector  $(x_1,\ldots,x_k)\in \mathbb{R}^k$
 is identified with the vector $(x_1,\ldots,x_k,0,0,\ldots)\in \mathbb{R}^\infty$ and $\mathcal{P}(\mathbb{R}^k)$
 is naturally identified with the subset of $\mathcal{P}(X)$ consisting of measures concentrated on~$\mathbb{R}^k$.

 It follows from our first theorem that for each $k$ there is a probability solution $\mu_k\in \mathcal{P}(\mathbb{R}^k)$
 to the nonlinear equation with the drift $v^k$, moreover, we have that such a solution is absolutely continuous
 with respect to the standard Gaussian measure $\gamma_k$ on~$\mathbb{R}^k$. Indeed, we can consider $v^k_n(\mu,x)$ only for
 absolutely continuous measures and these functions satisfy the continuity assumption from that theorem,
 because  weak convergence of densities in $L^2(\gamma_k)$ implies weak convergence of the corresponding measures.

 Let us take $V(x)=|x|_{_E}^2=\sum_{n=1}^\infty \alpha_n x_n^2$ as a Lyapunov function for our finite-dimensional
 equations. We have
 \begin{multline*}
 L_{v^k(\mu,\bullet)}V(x)=2\sum_{n=1}^k \alpha_n -2 \sum_{n=1}^k \alpha_n x_n^2+ 2 \sum_{n=1}^k \alpha_n x_n v_n^k(\mu,x)
 \\
 \le 2T -2V(x)+2C V(x)^{1/2}T^{1/2}
 \le 2T+C^2T-V(x).
  \end{multline*}
 Since $\mu_k$ satisfies a linear equation with an operator satisfying the above estimate with $V$, it follows
 from standard a priori estimates for solutions to linear Fokker--Planck--Kolmogorov equations (see \cite{bookFPK})
 that
 $$
 \int_{\mathbb{R}^k} V\, d\mu_k\le (2+C^2)T.
 $$
 The sets $\{V\le R\}$ are compact in $\mathbb{R}^\infty$, hence it follows from the Chebyshev inequality that
 the measures $\mu_k$ are uniformly tight on $\mathbb{R}^\infty$. Passing to a subsequence, we can assume that
 they converge weakly to some probability measure $\mu$ on $\mathbb{R}^\infty$.

 Let us verify that $\mu$ is a solution \eqref{eq1-n}. Let $\varphi\in C_0^\infty(\mathbb{R}^k)$. We can assume that
 $|\partial_{x_j}\varphi|\le 1$.
 For all $n\ge k$ we have
 $$
 \int_{\mathbb{R}^n} [\Delta \varphi(x) -(x,\nabla\varphi(x))+
 (v(\mu_n,x),\nabla\varphi(x))]\, \mu_n(dx)=0.
 $$
 Obviously,
 $$
 \int_{\mathbb{R}^n} [\Delta \varphi(x) -(x,\nabla\varphi(x))]\, \mu_n(dx)
 \to
 \int_{\mathbb{R}^\infty} [\Delta \varphi(x) -(x,\nabla\varphi(x))]\, \mu(dx)
 $$
 as $n\to\infty$. Let us show that for each $j\le k$ we
 have
 $$
 \int_{\mathbb{R}^n} v_j(\mu_n,x)\partial_{x_j}\varphi(x)\, \mu_n(dx)\to
 \int_{\mathbb{R}^\infty} v_j(\mu,x)\partial_{x_j}\varphi(x)\, \mu(dx).
 $$
 Let $\varepsilon>0$. We have $|v_j(\mu_n,x)\partial_{x_j}\varphi(x)|\le C$.
 There is $R>0$ such that
 $$
 \mu_n(V> R)+\mu(V> R) \le \varepsilon C^{-1}
 $$
 for all $n$. Hence the integrals of $v_j(\mu_n,x)\partial_{x_j}\varphi(x)$ over the set $\{V> R\}$
 do not exceed $\varepsilon$. The sequence $\{\mu_n\}$ with the added limit $\mu$ is a compact set.
 By the uniform continuity of $v_j$ on compacts sets we conclude that
 $v_j(\mu_n,x)\to v_j(\mu,x)$ uniformly on $\{V\le R\}$. Hence for all $n$ large enough we have
 $$
 \biggl|\int_{\{V\le R\}} v_j(\mu_n,x)\partial_{x_j}\varphi(x)\, \mu_n(dx)-
 \int_{\{V\le R\}} v_j(\mu,x)\partial_{x_j}\varphi(x)\, \mu_n(dx)\biggr|<\varepsilon.
 $$
 In addition, for all $n$ large enough
 $$
 \biggl|\int_{X} v_j(\mu,x)\partial_{x_j}\varphi(x)\, \mu_n(dx)-
 \int_{X} v_j(\mu,x)\partial_{x_j}\varphi(x)\, \mu(dx)\biggr|<\varepsilon.
 $$
 For such $n$ we obtain
 $$
 \biggl|\int_{\{V\le R\}} v_j(\mu,x)\partial_{x_j}\varphi(x)\, \mu_n(dx)-
 \int_{\{V\le R\}} v_j(\mu,x)\partial_{x_j}\varphi(x)\, \mu(dx)\biggr|< 2\varepsilon,
 $$
 hence
 $$
 \biggl|\int_{\{V\le R\}} v_j(\mu_n,x)\partial_{x_j}\varphi(x)\, \mu_n(dx)-
 \int_{\{V\le R\}} v_j(\mu,x)\partial_{x_j}\varphi(x)\, \mu(dx)\biggr|< 3\varepsilon.
 $$
 Therefore, we have the desired convergence of the integrals of $v_j(\mu_n,x)\partial_{x_j}\varphi(x)$, which shows
 that $\mu$ is a solution to our equation.
  \end{proof}

  The authors state no conflict of interest.

  We thank the anonymous referee for useful corrections.

This research is supported by the Russian Science Foundation Grant N 25-11-00007
at the Lomonosov Moscow State University.

\end{document}